\documentclass[10pt]{article}
\font\smallit=cmti10
\font\smalltt=cmtt10

\usepackage{amssymb,latexsym,amsmath,epsfig,amsthm} 

\makeatletter

\renewcommand\section{\@startsection {section}{1}{\z@}
{-30pt \@plus -1ex \@minus -.2ex}
{2.3ex \@plus.2ex}
{\normalfont\normalsize\bfseries\boldmath}}

\renewcommand\subsection{\@startsection{subsection}{2}{\z@}
{-3.25ex\@plus -1ex \@minus -.2ex}
{1.5ex \@plus .2ex}
{\normalfont\normalsize\bfseries\boldmath}}

\renewcommand{\@seccntformat}[1]{\csname the#1\endcsname. }

\makeatother

\newtheorem{theorem}{Theorem}
\newtheorem{lemma}{Lemma}

\newtheorem{corollary}{Corollary}

\theoremstyle{definition}
\newtheorem{definition}{Definition}

\newtheorem{remark}{Remark}

\def\multiset#1#2{\ensuremath{\left(\kern-.3em\left(\genfrac{}{}{0pt}{}{#1}{#2}\right)\kern-.3em\right)}}

\DeclareFontFamily{U}{wncy}{}
    \DeclareFontShape{U}{wncy}{m}{n}{<->wncyr10}{}
    \DeclareSymbolFont{mcy}{U}{wncy}{m}{n}
    \DeclareMathSymbol{\Sha}{\mathord}{mcy}{"58} 

\def \Q {{\mathbb Q}}
\def \F {{\mathbb F}}
\def \R {{\mathbb R}}
\def \Z {{\mathbb Z}}

\def \a {\alpha}
\def \d {\delta}
\def \g {\gamma}
\def \x {\xi}
\def \y {\eta}
\def \om {\omega}
\def \Om {\Omega}
\def \fp {{\mathfrak p}}
\def \im {\textup{im}}
\def \coker {\textup{coker}}
\def \phih {\hat{\phi}}
\def \psih {\hat{\psi}}
\def \Gal{\textup{Gal}}
\def \Sel {\textup{Sel}}
\def \Selphi {\Sel^{(\phi)}}
\def \Selpsi {\Sel^{(\psi)}}
\def \Selphih {\Sel^{(\phih)}}
\def \Selpsih {\Sel^{(\psih)}}
\def \WC{\textup{WC}}
\def \Qthree {\Q^\times/{\Q^\times}^3}
\def \Lthree {L^\times/{L^\times}^3}
\def \Lstar {L(S_{E,\psi},3)^*}

\usepackage{hyperref}
\usepackage{comment}
\usepackage{tikz-cd}
\usepackage[all]{xy}
\usepackage[bb = boondox]{mathalpha}

\begin{document}


\begin{center}
\uppercase{\bf A parametrized set of explicit elements of $\mathbf{\Sha(E/\mathbbb{Q})[3]}$}
\vskip 20pt
{\bf Steven R. Groen}\\
{\smallit Department of Mathematics, University of Warwick, Zeeman Building, Coventry, United Kingdom.}\\
{\tt steven.groen@warwick.ac.uk}\\ 
\vskip 10pt
{\bf Jaap Top}\\
{\smallit Bernoulli Institute, University of Groningen, Nijenborgh 9, Groningen, the Netherlands.}\\
{\tt j.top@rug.nl}\\
\end{center}
\vskip 20pt

\centerline{\smallit Received: 4/12/22, Revised: 9/5/22, Accepted: 11/20/22, Published: 12/16/22} 
\vskip 30pt 

\centerline{\bf Abstract}

\noindent
In this paper a set of elliptic curves $E$ with explicit elements of order $3$ in their Tate-Shafarevich group is constructed. First, the theory of descent by $3$-isogeny is reviewed, including explicit equations for homogeneous spaces representing the elements in the associated Selmer group. For the main result, elliptic curves admitting a rational $2$-isogeny as well as a rational $3$-isogeny are constructed. Using elementary $2$-isogeny descent, it is shown that our curves have rank zero. A result of Cassels then shows that the Selmer group of the $3$-isogeny is non-trivial. As a consequence one obtains in a very simple way explicit examples of plane cubics over $\Q$ that have a point everywhere locally, but not globally.

\pagestyle{myheadings}
\markright{\smalltt INTEGERS: 22 (2022)\hfill}
\thispagestyle{empty}
\baselineskip=12.875pt
\vskip 30pt



\section{Introduction}

In this article, we consider elliptic curves defined over $\Q$ that admit a rational isogeny of degree $3$. As is well known (see, for instance., \cite[Theorem~X.4.2]{BibSilv}) such an isogeny $\psi \colon E \to \bar{E}$ gives rise to an exact sequence of $\F_3$-vector spaces
\begin{equation*} \label{eqselsha1}
    0 \to \bar{E}(\Q)/\psi(E(\Q)) \to \Selpsi(E/\Q) \to \Sha(E/\Q)[\psi] \to 0 . 
\end{equation*}
The elements of $\Selpsi(E/\Q)$ are represented by genus-$1$ curves called \emph{homogeneous spaces}. As a preparation for our main result, section~\ref{sechomsp} presents explicit equations for these curves. The fact that they come from the Selmer group means they have points everywhere locally. A global point on such a curve gives rise to a point on $\bar{E}(\Q)/\psi(E(\Q))$. 
Section~\ref{sechomsp} is the ``{\sl Descent via Three-Isogeny}''
analogue of the classical ``{\sl Descent via
Two-Isogeny}''  described in \cite[Proposition ~X.4.9]{BibSilv}. Although most of this
is known to experts and parts of it can be found in the literature
(\cite{BibAoki}, \cite[Section~8.4]{Cohen}, \cite{BibCohPaz}, \cite{BibFish},
\cite{BibTop}), we are convinced our short and complete
exposition is useful.

In section~\ref{secfam} the main result of the paper is presented: it combines and compares the two simplest types of descent, namely a $2$-isogeny descent and a $3$-isogeny descent. We construct an explicit set of elliptic curves defined over $\Q$ (and admitting both a rational $2$-isogeny and a rational $3$-isogeny) with elements of order $3$ in their Tate-Shafarevich group. Concretely, we prove the following theorem.

\begin{theorem} \label{thmmain}
Let $h$ be a positive integer such that $h \equiv 3 \bmod 8$ and moreover $h$, $h-2$, $h-6$ and $h-8$ are prime numbers. Define
$$ E_h\colon y^2= x^3-216(x-h(h-6)^2)^2.$$
Then $E_h$ has the following properties:
\begin{enumerate}
    \item $E_h$ has a rational $2$-isogeny $\phi$ and a rational $3$-isogeny $\psi$ (see section~\ref{subsecEhdef}).
    \item $E_{h,tors}(\Q) \cong \Z/2\Z$ (see section~\ref{subsecEhdef}).
    \item Using $2$-isogeny descent one concludes that $E_h(\Q)$ has rank zero (see section~\ref{subsecEh2-desc}).
    \item $\#\Selpsi(E_h/\Q)=9$; this gives rise to $9$ pairwise distinct elements in the group $\Sha(E_h/\Q)[3]$ (see section~\ref{subsecEh3-desc}).
\end{enumerate}
\end{theorem}

By section~\ref{sechomsp}, this yields explicit cubics violating the Hasse principle. For example
\[
    C^h\colon 3w^2z+2z^3+w^3+6wz^2+2(h-2)^2(h-8)=12z^2-6w^2
\] 
 has a point over every completion of $\Q$, but not over $\Q$ itself.

Schinzel's hypothesis H (\cite{BibSchinzel}) 
predicts as a specific case that
\[
S=\left\{ h\in \Z \;:\; h\equiv 3\bmod 8\;\text{and}\; h,\;h-2,\;h-6,\;h-8\;\text{prime}\right\}
\]
is infinite and hence that this set yields infinitely many examples as above.

An infinite family of genus $1$ curves violating the Hasse principle and related to $\Sha[3]$ is already given in \cite{Poonen}, using the Brauer-Manin obstruction to show the non-existence of rational points (the earliest examples coming from $\Sha[3]$ seem to be those from Selmer's paper \cite{Selmer}). In \cite{bhargava2014positive} it is shown that in fact a positive proportion of plane cubics fail the Hasse principle. Furthermore, it is a classic result (\cite{BibCAss2}, \cite{BibAoki}, \cite{FisherPlatonic}) that the $3$-torsion in Tate-Shafarevich groups of elliptic curves over
$\Q$ gets arbitrarily large. In contrast with our result, apart from
\cite{Poonen} and \cite{bhargava2014positive} the papers
cited here discuss elliptic curves admitting
a rational point of order $3$. The method of comparing $2$- and $3$-isogeny descents we used
is very natural, yet we are not aware of an earlier text presenting examples in this way.  

\section{Homogeneous spaces corresponding to a rational \texorpdfstring{$3$}\     -isogeny} \label{sechomsp}

In order to construct explicit counterexamples to the Hasse principle in section~\ref{secfam}, this section reviews via Galois cohomology the descent coming from a rational $3$-isogeny $\psi$. This builds up to Theorem~\ref{thmhomspeq}, giving explicit equations for the homogeneous spaces in $\WC(E/\Q)$ coming from $\psi$. As already
stated, our aim is to provide
a description for $3$-isogeny descent comparable to what is done for $2$-isogenies in
\cite[Proposition ~X.4.9]{BibSilv}.

\subsection{Elliptic curves with a rational \texorpdfstring{$3$}\ -isogeny}\label{3isogenyequations}

Let $E$ be an elliptic curve over $\Q$ with a Galois invariant subgroup $T$ of size $3$. The quotient map $\psi\colon E \to E/T =: \bar{E}$ is a rational $3$-isogeny. By, for example, \cite[Proposition~ III.4.12]{BibSilv},
every rational $3$-isogeny between elliptic curves is obtained in
this way. 
Following \cite{BibTop}, one distinguishes the cases $j_E=0$ and $j_E \neq 0$ to obtain explicit descriptions for these curves and maps. 
The nice discussion of $3$-isogeny descent in \cite[Section~8.4]{Cohen}
avoids this distinction by using equations $y^2=x^3+d(ax+b)^2$ (but
different from our text, it does not discuss the Galois cohomological derivation of various
maps, nor the form of various homogeneous spaces).

In the case $j_E=0$, one gives $E$  by 
\begin{equation*} \label{eqEj0}
E\colon y^2=x^3+A
\end{equation*}
(here $T$ is generated by a point with $x=0$), and the isogeny $\psi$ up to possibly multiplication by $-1$ is 
\begin{align}
    \psi&: (x,y) \mapsto \left(\frac{y^2+3A}{x^2} \;, \;\frac{y(x^3-8A)}{x^3}  \right) , \label{eqpsiC} \nonumber \\
    \bar{E}&:\y^2=\xi^3-27A  . \nonumber
\end{align}
For brevity, we introduce the notation $\bar{A}:=-27A$.

In the case $j_E \neq 0$, one gives $E$ by
\begin{equation*} \label{eqEj}
E\colon y^2=x^3+A(x-B)^2. 
\end{equation*}
Again $T$ is generated by a point with $x=0$, and (up to $\pm 1$) here $\psi$  is given by
\begin{align}
    \psi&:(x,y) \mapsto \left(\frac{3(6y^2+6AB^2-3x^3-2Ax^2)}{x^2} \;, \; \frac{27y(8AB^2-x^3-4ABx)}{x^3} \right) , \label{eqpsiAB} \\
    \bar{E}&:\y^2= \x^3-27A(\xi-4A-27B)^2  . \nonumber 
\end{align} 
For brevity, we introduce $\bar{B}:=4A+27B$.

By $\psih\colon \bar{E} \to E$ we will denote the dual isogeny of $\psi$, such that $\psih \circ \psi = [3]$.

\subsection{Galois cohomology}

For any field $K$, let $G_K=\Gal(K^{\text{sep}}/K)$ be its absolute Galois group. The isogeny $\psi$ described above yields an exact sequence 
\begin{equation} \label{exseqpsi}
\begin{tikzcd}
0 \arrow[r] & {E[\psi]} \arrow[r] & E \arrow[r, "\psi"] & \bar{E} \arrow[r] & 0
\end{tikzcd}
\end{equation}
of $G_\Q$-modules. Taking $G_\Q$ invariants yields a long sequence of Galois cohomology
\[\begin{tikzcd} 0 \arrow[r] & E(\Q)[\psi] \arrow[r]& E(\Q) \arrow[r,"\psi"]\arrow[d,phantom, ""{coordinate, name=Z}]& \bar{E}(\Q) \arrow[dll,"\delta" description,rounded corners,to path={ --([xshift=2ex]\tikztostart.east)|- (Z)[near end]\tikztonodes-| ([xshift=-2ex]\tikztotarget.west)-- (\tikztotarget)}] \\ & H^1(G_\Q,E[\psi]) \arrow[r]& H^1(G_\Q,E) \arrow[r,"\psi_*"]& H^1(G_\Q,\bar{E}), \end{tikzcd}\]
which is shortened to 
\begin{equation} \label{eqpsigal} 
0 \to \bar{E}(\Q)/\psi(E(\Q))\stackrel{\delta}{\longrightarrow} {H^1(G_\Q,E[\psi])} \longrightarrow  {H^1(G_\Q,E)[\psi_*]} \to 0.
\end{equation}

Recall that a \emph{homogeneous space} (also called a torsor) of $E/\Q$ is a smooth curve $C/\Q$ with a simply transitive action $E \times C \to C$.  By \cite[Theorem~X.3.6]{BibSilv}, elements of $H^1(G_\Q,E)$ are in bijection with $\Q$-isomorphism classes of homogeneous spaces. Hence the set of $\Q$-isomorphism classes of homogenous spaces inherits the structure of a group, called the Weil-Ch\^{a}telet group $\WC(E/\Q) \cong H^1(G_\Q,E)$. In this section we explicitly describe the homogeneous spaces coming from a certain subgroup of
$H^1(G_\Q,E)$.

Localizing \eqref{eqpsigal} gives the commutative diagram with exact rows
{\small\begin{equation} \label{diagloc}
\begin{tikzcd}
\bar{E}(\Q)/\psi(E(\Q)) \arrow[d] \arrow[r,hook,"\delta"]       & {H^1(G_\Q,E[\psi])} \arrow[r, two heads] \arrow[d]   & {\WC(E/\Q)[\psi]} \arrow[d] \\
\prod\limits_{v \in M_\Q} \bar{E}(\Q_v)/\psi(E(\Q_v)) \arrow[r,hook, "\delta"] & \prod\limits_{v \in M_\Q} H^1(\Gal(\bar{\Q_v}/\Q_v),E[\psi]) \arrow[r, two heads] & \prod\limits_{v \in M_\Q} \WC(E/\Q_v) [\psi] .
\end{tikzcd}
\end{equation}
}

This allows one to define Selmer groups and Tate-Shafarevich groups, recalled here.

\begin{definition} \label{defsel}
The \emph{$\psi$-Selmer group} of $E$ over $\Q$, denoted $\Selpsi(E/\Q)$, is defined as the kernel of the composite map $H^1(G_\Q,E[\psi])\to {\prod_{v \in M_\Q} \WC(E/\Q_v)[\psi]}$ in the diagram \eqref{diagloc}.
\end{definition}

\begin{definition} \label{defsha}
The \emph{Tate-Shafarevich group} of $E$ over $\Q$, denoted $\Sha(E/\Q)$, is the kernel of the localization map  $\WC(E/\Q) \to \prod_{v \in M_\Q}  \WC(E/\Q_v).$
\end{definition}

From \eqref{diagloc} one infers the exact sequence 
\begin{equation*} \label{eqselsha2}
    0 \to \bar{E}(\Q)/\psi(E(\Q)) \to \Selpsi(E/\Q) \to \Sha(E/\Q)[\psi] \to 0 .
\end{equation*}

\subsection{Explicit equations for homogeneous spaces}

Define $L:=\Q[T]/(T^2-\bar{A})$ and denote $\bar{\tau}:=T\bmod(T^2-\bar{A})\in L$. The points in $\bar{E}[\hat{\psi}]$ have coordinates in any field obtained as the image
of $L$ under an algebra homomorphism. 

In case $L$ is a field, one
has $E[\psi] \cong \mu_3$ as $G_L$-modules. Since $\#E[\psi]^{G_L}$
divides $3$ and hence is coprime to $\#\Gal(L/\Q)$, all cohomology groups $H^k(\Gal(L/\Q),E[\psi]^{G_L})$ are trivial.
Therefore (using the inflation - restriction exact sequence,
as is done in \cite[Section~4]{BibTop}) the restriction $H^1(G_\Q,E[\psi]) \to H^1(G_L,E[\psi])^{\Gal(L/\Q)}$ is an isomorphism. Combining this with Hilbert's Theorem 90 and writing 
$\text{Ker}(N_{L/\Q}):=\text{Ker}\left(N_{L/\Q}\colon L^\times/{L^\times}^3\to\Q^\times/{\Q^\times}^3\right) $ gives
\[ H^1(G_\Q,E[\psi]) \cong H^1(G_L,E[\psi])^{\Gal(L/\Q)} \cong H^1(G_L, \mu_3)^{\Gal(L/\Q)} \cong \text{Ker}(N_{L/\Q}).\] 
Under this sequence of isomorphisms, \eqref{eqpsigal} becomes
\begin{equation} \label{eqpsigalL} 
    \begin{tikzcd}
0 \arrow[r] & \bar{E}(\Q)/\psi(E(\Q)) \arrow[r,"\delta"] & 
\text{Ker}(N_{L/\Q}) \arrow[r] & \WC(E/\Q)[\psi] \arrow[r] & 0.
\end{tikzcd}
\end{equation}

   In case $L$ is {\em not} a field one has $L\cong \Q\times \Q$.
   Under this identification the norm map $N_{L/\Q}\colon L^\times/{L^\times}^3\to\Q^\times/{\Q^\times}^3$ yields
   multiplication $\Q^\times/{\Q^\times}^3\times\Q^\times/{\Q^\times}^3\to\Q^\times/{\Q^\times}^3$. In particular, its kernel
   $\text{Ker}(N_{L/\Q})$ is isomorphic to $\Q^\times/{\Q^\times}^3$.
   Since in the present situation $E[\psi]\cong \mu_3$ as $G_\Q$-modules, one obtains in an analogous but simpler way as
   above again $H^1(G_\Q,E[\psi]) \cong \text{Ker}(N_{L/\Q})$
   and therefore also the sequence \eqref{eqpsigalL}.
   
   The following lemma describes the connecting homomorphism $\delta$
   in terms of the equations presented in Section~\ref{3isogenyequations}.
\begin{lemma} \label{lemdel}
Let $E$ be given by $y^2=x^3+A$. The map $\delta\colon \bar{E}(\Q)/\psi(E(\Q)) \to \Lthree$ is given by
\[
    \delta\colon 
    P+\psi(E(\Q)) \mapsto \left\{\begin{array}{ll}
    1 \cdot {L^\times}^3 & \text{if}\;\;P=O,\\
        \left(\y+{\bar{\tau}}\right) \cdot {L^\times}^3 & \text{if}\;\;P=(\x,\y)\;\;\text{with}\;\;\x\neq 0,\\
     (2A^2,4A)^{\pm 1}\in\left(\Qthree\right)^{\oplus 2} &\text{if}\;\;P=(0,\pm \sqrt{\bar{A}}).
    \end{array}\right.
\]

For $E$ given by $y^2=x^3+A(x-B)^2$,  the connecting homomorphism takes the form
\[
    \delta\colon P+\psi(E(\Q)) \mapsto \left\{\begin{array}{ll}
    1 \cdot {L^\times}^3 & \text{if}\;\;P=O,\\
    (\y+(\x-\bar{B})\bar{\tau}) \cdot  {L^\times}^3 &
    \text{if}\;\;P=(\x,\y)\;\;\text{with}\;\;\x\neq 0,\\
    (2\bar{B}^2A, 4\bar{B}A^2)\in\left(\Qthree\right)^{\oplus 2} &\text{if}\;\;P=(0,\pm \sqrt{\bar{A}}\cdot\bar{B}).
    \end{array}\right.
\]
\end{lemma}

\begin{proof}
The map $\delta\colon \bar{E}(\Q)/\psi(\Q)\to H^1(G_\Q,E[\psi])$ sends the class of $P$ in the group $\bar{E}(\Q)/\psi(E(\Q))$ to the cocycle $\sigma \mapsto Q^\sigma -Q$, where $\psi(Q)=P$. One finishes the proof by a direct computation using the definition of $\psi$ and Hilbert's Theorem 90. See also \cite[Exercise 10.1]{BibSilv} or \cite[\S8.4.4]{Cohen}.
\end{proof}

In case $L$ is a field, write $M_L$ for the set of primes (finite and infinite) of $L$. Let $S\subset M_L$ denote a finite set containing all infinite primes of $L$.
We define
\begin{align*}
    L(S,3) &:= \{t \in L^\times/L^{\times 3} \; : \; v(t) \equiv 0 \bmod 3\; \forall v \in M_L\setminus S\} \\
    L(S,3)^* &:= \{t \in L(S,3) \; : \;  N_{L/\Q}(t) \in {\Q^\times}^3\} . 
\end{align*}
Let $S_{E,\psi}\subset M_L$ denote the set of bad primes of $E$, together with the primes dividing $3$ (which is the degree of $\psi$), and the infinite primes of $L$. It is well-known (see, for instance, \cite[Corollary~X.4.4]{BibSilv}) that $\im(\delta)$ is contained in $L(S_{E,\psi},3)$. This implies that $\im(\delta)$ is finite, which is essentially the weak Mordell-Weil theorem. Note that it is immediate from the equations (without using Galois cohomology) that $\im(\delta)$ is contained in the kernel of the norm map.

To adapt the above to the situation where $L$ is not a field (so
$\bar{A}$ is a non-zero square), one considers in that case any of the two projections
$\pi\colon \Lthree\to\Qthree$ and defines, for $S\subset M_\Q$
containing the infinite prime of $\Q$, the group
\[
L(S,3) := \{t \in L^\times/L^{\times 3} \; : \; v(\pi(t)) \equiv 0 \bmod 3\; \forall v \in M_\Q\setminus S\}.
    \]
    The subgroup $L(S,3)^*$ is defined exactly as before.
    With $S_{E,\psi}\subset M_\Q$ the set of bad
    primes for $E$ together with $3$ and $\infty$,
    again the image of $\delta$ is contained in $L(S_{E,\psi},3)^*$.

\begin{lemma} \label{lemcube}
Suppose $L$ is a field and $t \in L(S,3)^*$. If $v\in S$ is a
finite place with the property that $v(t)\bmod 3\not\equiv 0$, then
$v$ is either a split prime of $L/\Q$ or a non-principal ramified prime.
\end{lemma}
\begin{proof}
Suppose $v\in S$ corresponds to an inert or a principal ramified prime, which we write as 
$\fp=\a \mathcal{O}_L$. Then $N_{L/\Q}(\a)\in\{\pm p, \pm p^2\}$ for a prime number $p$. If $v(t)\bmod 3\not\equiv 0$,
then any $\tau\in L^\times$ representing $t$ has a norm containing
the prime $p$ to a power that is not $0\bmod 3$. This violates
the definition of $L(S,3)^*$.
\end{proof}

We now give explicit equations for the homogeneous space corresponding to such $t$.

\begin{theorem} \label{thmhomspeq}
Let $t\in L$ represent a class in $\Lstar$ and let $s\in\Q$ be the
unique rational number with $N_{L/\Q}(t)=s^3$. The map $\Lstar \to \WC(E/\Q)[\psi]$ in \eqref{eqpsigalL} sends $t\cdot {L^\times}^3$  to a homogeneous space $C_t$. An explicit affine equation for $C_t$ in variables $w$ and $z$ is as follows:
\begin{itemize}
\item If $E$ is given by $y^2=x^3+A$ and $\bar{A}$ is a square, then $C_t\colon t^2w^3-2t\sqrt{\bar{A}}=z^3$.
\item If $E$ if given by $y^2=x^3+A$ and $\bar{A}$ is not a square, write $t=u+v\sqrt{\bar{A}}$. 
Then we have $$ C_t\colon 3uw^2z+\bar{A}uz^3+vw^3+3\bar{A}=1. $$
\item For $E\colon y^2=x^3+A(x-B)^2$ and $\bar{A}$ a square, we have $C_t\colon t^2w^3-2t(wz-\bar{B})\sqrt{\bar{A}}=z^3$.
\item For $E\colon y^2=x^3+A(x-B)^2$ and $\bar{A}$ not a square, write $t=u+v\sqrt{\bar{A}}$. Then
$$C_t\colon 3uw^2z+\bar{A}uz^3+vw^3+3\bar{A}vwz^2+\bar{B}=s(w^2-\bar{A}z^2).$$
\end{itemize}
\end{theorem}

\begin{proof}
Assume that $t$ comes from  $P=(\x,\y)$ on $\bar{E}(\Q)/\psi(E(\Q))$, so $t$ equals the image of $\delta$ in Lemma~\ref{lemdel} up to multiplication by a cube. For every case this yields a model of $C_t$.

In the first case, $t \in \im (\delta)$ implies $w^3 t = \y+\sqrt{\bar{A}}$ for some $w \in \Q$. The condition $P\in\bar{E}(\Q)$ yields that $(w^3t-\sqrt{\bar{A}})^2=\y^2=\x^3+\bar{A}$. Setting $z=\frac{\x}{w}$ gives the equation of $C_t$.

The second case is similar, starting from 
$(w+z\sqrt{\bar{A}})^3 (u+v\sqrt{\bar{A}}) = \y+\sqrt{\bar{A}}$. Comparing coefficients of $\sqrt{\bar{A}}$ yields the equation of $C_t$. 
Here $P$ is automatically on the curve by letting $\x^3$ be the norm of the left-hand side (which is a cube).

For the third case, $t \in \im (\delta)$ implies $w^3t = \y+\sqrt{\bar{A}}(\x-\bar{B})$. Just like in the first case, the condition $P\in\bar{E}(\Q)$ gives the equation for $C_t$.

In the fourth case, we have
 $(w+z\sqrt{\bar{A}})^3 (u+v\sqrt{\bar{A}}) = \y+\sqrt{\bar{A}}(\x-\bar{B})$.
Comparing coefficients of $\sqrt{\bar{A}}$ yields 
$3uw^2z+\bar{A}uz^3+vw^3+3\bar{A}vwz^2+B= \x$.
Now, observe that 
\[
\x^3= N_{L/\Q}(\y+\sqrt{\bar{A}}(\x-\bar{B})) 
=N_{L/\Q}((w+z\sqrt{\bar{A}})^3 (u+v\sqrt{\bar{A}}))=s^3(w^2-\bar{A}z^2)^3. 
\]
Substituting $\x=s(w^2-\bar{A}z^2)$ gives the equation for $C_t$.
\end{proof}

\begin{remark}
Theorem 3.1 and Theorem 4.1 of \cite{BibCohPaz} also describe the homogeneous spaces of a rational $3$-isogeny. Like in \cite{Cohen}, the slightly different model $\bar{E}\colon y^2=x^3+D(ax+b)^2$ is used. Theorem 3.1 treats  $\sqrt{D}\in\Q$, resulting in
$u X^3+ \left(\frac{1}{u}\right) Y^3 +2bZ^3=2aXYZ$.
This is equivalent to our first case (if $a=0$) and third case (if $a \neq 0$), via the change of variables $(X,Y,Z):=(w,-z,1)$. 
Theorem~4.1 of \cite{BibCohPaz} treats the case $\sqrt{D}\not\in \Q$, corresponding to the second case (if $a=0)$ and fourth case (if $a \neq 0$) in Theorem \ref{thmhomspeq}. Introducing an extra variable for the cube root of $N_{L/\Q}(t)$ is avoided by writing $\Gal(L/\Q)=\langle \sigma \rangle$ and $t=v^2\sigma(v)$ (which is allowed because $t$ is chosen up to cubes), such that $N_{L/\Q}(t)=(v\sigma(v))^3$. This yields a different, but equivalent model of the corresponding homogeneous space.
\end{remark}

\section{A set of explicit elements of order \texorpdfstring{$3$}\ \  in Tate-Shafarevich groups} \label{secfam}

In this section we exhibit a set $\{E_h\}$ of elliptic curves with a rational $2$-isogeny $\phi$ and a rational $3$-isogeny $\psi$. By the well-known method of $2$-isogeny descent we show that every $E_h$ has rank zero. Then we show that the Selmer group $\Selpsi(E_h/\Q)$ must contain at least $9$ elements. 
The non-trivial ones are represented by homogeneous spaces (in fact plane cubics whose model is given by Theorem~\ref{thmhomspeq}), that violate the Hasse principle. We guarantee the existence of elements in the $\psi$-Selmer group using the following result of Cassels (\cite[Thm.~1.1]{BibCass}, \cite[Thm.~1]{KloosSchaef}, \cite[\S~7]{DeLong}).

\begin{theorem} \label{thmCass}
Let $\phi :E/\Q \to E'/\Q$ be a rational isogeny and let $\phih:E'/\Q \to E/\Q$ be its dual.
Then the following relation between the corresponding Selmer groups holds (\cite[Theorem~1.1]{BibCass}):
$$\frac{\# \Selphih(E'/\Q)}{\# \Selphi (E/\Q)} = \frac{\# E'(\Q)[\phih] \Om_E \prod_{p} c_{E,p} } { \# E(\Q)[\phi] \Om_{E'} \prod_{p} c_{E',p} },$$
where $\Om_E:= \int_{E(\R)} \om_E$ denotes the real period of $E$ and $c_p:=\#E(\Q_p)/E_0(\Q_p)$ denotes the Tamagawa number of $E$ at $p$. 
\end{theorem}

\subsection{The elliptic curves \texorpdfstring{$E_h$}{Eh}} \label{subsecEhdef}

Recall the general form of an elliptic curve with a rational $3$-isogeny and non-zero $j$-invariant:
$$ E\colon y^2=x^3+A(x-B)^2.$$ We wish to use descent by $2$-isogeny to show that $E$ has rank zero, so we want to impose on the parameters $A$ and $B$ the condition that $E$ admits a rational $2$-isogeny. This is equivalent to the right-hand side polynomial having a rational root $x=c$. This implies $A=-\frac{c^3}{(c-B)^2}$. Multiplying by $(c-B)^6$, rescaling $x$ and $y$ and introducing $h:=B$ yields
$$ E\colon y^2 = x^3-c^3(x-h(h-c)^2)^2. $$
We now make a strategic choice of $c$; we will later do arithmetic in $\Q(\bar{E}[\psih])=\Q(\sqrt{3c})$, so we pick $c=6$ to land in the convenient number field $\Q(\sqrt{2})$. This gives the model of the curve we will use:
\begin{equation} \label{eqEh}
E_h \colon y^2= x^3-216(x-h(h-6)^2)^2.
\end{equation}
We denote the rational $2$-isogeny by $\phi\colon E_h \to E_h'$, and the rational $3$-isogeny by $\psi\colon E_h \to \bar{E}_h$.

\begin{remark}
Viewing $h$ as a variable, one can view $E_h/\Q(h)$ as the generic
fiber of an elliptic surface $\mathcal{E}\to\mathbb{P}^1$ defined over
$\Q$. This turns out to be a quadratic twist of Beauville's
rational modular elliptic surface associated to the congruence
subgroup $\Gamma_0(6)$. Indeed, Beauville in \cite{Beauville} describes this
surface as the pencil of plane cubics given by
\[(x+y)(x+z)(y+z)+txyz=0.\] Using for instance the basepoint/section $(0:0:1)$ on
the curves in this pencil, one readily obtains the Weierstrass equation
\[
y^2=x^3+((t+2)x+4t)^2
\]
for the latter surface. The M\"{o}bius transformation
$h\mapsto  t=-2h/(h-6)$ transforms this, after scaling,
into
\[
y^2=x^3+(6x+h(h-6)^2)^2
\]
which is the quadratic twist over $\Q(h)(\sqrt{-6})/\Q(h)$ of 
$E_h$.

As a consequence (see also \cite[\S~2.3.2]{TopYui} and No.~66 in
the table of \cite{OguisoShioda}), as an
elliptic curve over the function field $\Q(h)$ one finds
$E_h(\Q(h))\cong \Z/2\Z$ and $E_h(\Q(h)(\sqrt{-6}))\cong\Z/6\Z$.
In \cite[p.~167 Table~8.3]{SchuttShioda} another equation
defining Beauville's example is presented.
\end{remark}

We now consider values $h\in\Z$ and take $E=E_h/\Q$. In order to maintain some control over the local computations, we minimize the number of bad primes of $E$; we make sure that the discriminant
$$ \Delta_E =  -2^{10}  3^9 h^3 (h- 2)^2 (h - 6)^6 (h - 8) $$
has only $6$ bad primes, by demanding that $h$, $h-2$, $h-6$ and $h-8$ are prime numbers. Note that this implies $h \equiv 4 \bmod 15$. 

From reduction modulo $5$, which we have rigged to be a good prime, and the fact that $-6$ and $2$ are not squares (meaning the $3$-torsion is not rational), it follows that the rational torsion subgroups of $E_h$, $E'_h$ and $\bar{E}_h$ are isomorphic to $\Z/2\Z$, as claimed in Theorem~\ref{thmmain}.

In view of Theorem~\ref{thmCass}, we wish to manipulate the Tamagawa numbers such that $ \prod_{p} c_{E,p} <  \prod_{p} c_{\bar{E},p}$ and hence $\Selpsi(E/\Q)$ grows. Using the algorithm in \cite{BibTateAlg} one obtains the results described in Table~\ref{tabloc}.

\begin{table}[ht]
\centering
\begin{tabular}{|l|l|l|l@{}|l@{}|} \hline
                       & $h \equiv 1 \bmod 8$                                                         & $h \equiv 3 \bmod 8$                                                        & $ h \equiv 5 \bmod 8$                                                        & $h \equiv 7 \bmod 8$                                                        \\ \hline
$p=2$                       & Additive                                                  & Additive                                              & Additive                                                  & Additive                                                 \\ \hline
$p=3$                    & Additive                                                 & Additive                                                 & Additive                                                & Additive                                                \\ \hline
$p=h-8$                     & Split                                                        & Non-split                                                  & Non-split                                                   & Split                                                      \\ \hline
$p=h-6$                     & Non-split                                                   & Non-split                                                   & Split                                                       & Split                                                      \\ \hline
$p=h-2$                     & Split                                                       & Split                                                       & Non-split                                                   & Non-split                                                  \\ \hline
$p=h$                      & Split                                                      & Non-split                                                   & Non-split                                                  & Split                                                      \\ \hline
$\prod_p c_{E_h,p}$       & $48$        & $ 16$       & $48$       & $144$      \\ \hline
$\prod_p c_{E_h',p}$      & $48$       & $ 16$      & $48$      & $144$     \\ \hline
$\prod_p c_{\bar{E}_h,p}$ & $144$ & $ 48$ & $ 16$ & $ 48$ \\ \hline
\end{tabular}
\caption{A table representing the local properties of $E_h$. `Split' and `non-split' refer to multiplicative reduction.}
 \label{tabloc}
\end{table}

The cases $h\equiv 1,3 \bmod 8$ are favorable in view of Theorem \ref{thmCass}. In order to keep the local computations required for descent by $2$-isogeny manageable, it helps to have small Tamagawa numbers, so we pick the case $h \equiv 3 \bmod 8$. Combining this with the earlier congruence relation $h \equiv 4 \bmod 15$ yields $h \equiv 19 \bmod 120$, which we assume from here on out. 

\subsection{Descent by \texorpdfstring{$2$}\  -isogeny} \label{subsecEh2-desc}

Let $E_h$ be given by \eqref{eqEh} with $h\equiv 3 \bmod 8$ and $h$, $h-2$, $h-6$ and $h-8$ primes. In order to apply the classical theory of $2$-isogeny descent, we move the $2$-torsion point to $(0,0)$. This yields the models
\begin{align}
    E_h \colon y^2&=x^3+\left(18(h-6)^2-216\right)x^2+108(h-2)(h-6)^3x \label{eqEh2} \\
    E_h' \colon Y^2&=X^3+\left(432-36(h-6)^2\right)X^2-108h^3(h-8)X. \label{eqEh'}
\end{align}
One verifies that both these models are globally minimal. The algorithm in \cite{BibTateAlg} is needed to show that this model of $E_h'$ is minimal at $2$ (we have $\text{ord}_2(\Delta_{E_h'})=14$), while in the other instances the computation of the discriminant suffices.

\begin{lemma} \label{lemEh2desc}
$E_h'(\Q)/\phi(E_h(\Q))=\{O,(0,0)\}$.
\end{lemma}
\begin{proof}
Following the classical descent by $2$-isogeny in \cite[Proposition~X.4.9]{BibSilv}, there is an injective homomorphism 
\begin{align*}
\a\colon E_h'(\Q)/\phi(E(\Q)) &\to \Q^\times / {\Q^\times}^2  \\
(X,Y) &\mapsto X \cdot {\Q^\times}^2 \;\;\;(\text{for}\;\;X\neq 0),\\
(0,0) &\mapsto -108h^3(h-8) \cdot {\Q^\times}^2 = -3h(h-8) \cdot {\Q^\times}^2.
\end{align*}
Hence $$\langle-3h(h-8)\rangle \subseteq \im(\a) \subseteq \Selphi(E/\Q)\subseteq \langle -1,2,3,h,h-2,h-6,h-8 \rangle.$$ We prove that $\Selphi(E/\Q)=\{1,-3h(h-8)\}$ by localizing to various primes.

Note that the square-free part of $X$ has to divide $-108h^3(h-8)$ in order to yield a point $(X,Y) \in E_h'(\Q)$ (see \cite[page~86-87]{BibTateSilv} for a more elaborate argumentation). This gives $\im(\a) \subseteq \langle -1,2,3,h,h-8 \rangle$.

For any prime $p$, we let $\a_p: E_h'(\Q_p)/\phi(E(\Q_p)) \to \Q_p^\times / {\Q_p^\times}^2 $ denote the localization of $\a$ at $p$. We first show that $E_{h,1}'(\Q_p)$ does not contribute. For odd $p$ this is immediate as $p\Z_p/2p\Z_p \cong 0$. For $p=2$, this is a straightforward calculation using the Laurent series in \cite[page 118]{BibSilv}. See also \cite[Lemma 3.2.1]{TOP1993211} for a similar calculation.

First, we analyze $$\a_2\colon E_h'(\Q_2)/\phi(E_h(\Q_2)) \to \Q_2^\times / {\Q_2^\times}^2= \langle -1,5,2\rangle.$$ Using our computation of Tamagawa numbers and the minimality of the model, we infer $\# (E'(\Q_2)/E_0'(\Q_2))=2$ (note that to ease notation, here and once more below we
removed the subscript $h$ in $E_h'$).
The reduction of $(X,Y)$ is singular precisely when $X$ is even, in which case it reduces to $(0,0)$. If $X$ is odd, then so is $Y$ and they satisfy the homogeneous equation
\begin{align*}
Y^2Z &=X^3+\left(432-36(h-6)^2\right)X^2Z-108h^3(h-8)XZ^2,
\intertext{implying}
Z&\equiv X+4Z+4X \equiv X \bmod 8.
\end{align*}
Hence $ X/Z = 1 \cdot {\Q_2^\times}^2 $. Moreover, observe that $$\a_2((0,0))=-108h^3(h-8) \cdot {\Q_2^\times}^2 = 5 {\Q_2^\times}^2,$$ and hence $\im(\a_2)=\langle 5 \rangle$. This gives the restriction $\im(\a) \subseteq \langle -3, -h,-(h-8)\rangle$.

Now consider the localized map $$\a_3:E_h'(\Q_3)/\phi(E_h(\Q_3)) \to \Q_3^\times / {\Q_3^\times}^2=\langle -1,3 \rangle.$$ Again we have computed $\#(E'(\Q_3)/E_0'(\Q_3))=2$.
The point $(X,Y)$ reduces to a singular point if and only if $3|X$, and otherwise $X \equiv 1 \bmod 3$. Moreover, $(0,0)$ is mapped to $-3h(h-8) \cdot {\Q_3^\times}^2 = 3 {\Q_3^\times}^2$. We conclude that $\im(\a_3)=\langle 3 \rangle$, giving the restriction $\im(\a) \subseteq \langle 3h, -(h-8) \rangle$.

The final restriction comes from the localized map
$$\a_{h-2}:E_h'(\Q_{h-2})/\phi(E_h(\Q_{h-2})) \to \Q_{h-2}^\times / {\Q_{h-2}^\times}^2=\langle -3,h-2 \rangle.$$ We see that $-3h(h-8) \cdot {\Q_{h-2}^\times}^2 = 36 \cdot {\Q_{h-2}^\times}^2 = 1 \cdot {\Q_{h-2}^\times}^2.$
The algorithm in \cite{BibTateAlg} shows that all points of $E_h'(\Q_{h-2})$ have good reduction. Over $\F_{h-2}$ the curve $E_h'$ is given by
$$E_h'\colon Y^2= X(X-72)^2. $$
Following \cite[Proposition~III.2.5]{BibSilv}, let $\g=\sqrt{72}\in \F_{h-2}$ be the slope of a tangent line at the node of $E_h'(\F_{h-2})$. The isomorphism
\[
    E_{h,ns}'(\F_{h-2}) \xrightarrow{\sim} \F_{h-2}^\times\;\; : \quad 
    (X,Y) \mapsto \frac{Y + \g (X-72)}{Y-\g(X-72)}
\]
maps the point $(18,27 \g)$ to $-3$, which is not a square, so this point is not a multiple of $2$. Finally, $18$ is a square in $\Q_{h-2}$, so we infer $\im(\a_{h-2})=\{1\}$. This gives the final restriction
$$ \im(\a)= \Selphi(E/\Q) = \{1,-3h(h-8)\}$$
and therefore $$ E_h'(E/\Q)/\phi(E_h(\Q)) = \{O,(0,0)\},$$
as desired.
\end{proof}

\begin{lemma} \label{lemEh'2desc}
$E_h(\Q)/\phih(E_h'(\Q))=\{O,(0,0)\}$.
\end{lemma}
\begin{proof}

We make use of Theorem \ref{thmCass}. Clearly $\# E_h'(\Q)[\phih]=\#E_h(\Q)[\phi]=2$. Furthermore, Table \ref{tabloc} shows $\prod_p c_{E_h,p}=\prod_p c_{E_h',p}=16$. For the real periods, we make use of \cite[Lemma 7.4]{BibDokLoc}:
\begin{equation} \label{eqreaper}
\frac{\Om_{E_h}}{\Om_{E_h'}} = \frac{\# \textup{ker} \left( \phi\colon E_h(\R) \to E_h'(\R) \right) }{\#\coker \left( \phi\colon E_h(\R) \to E_h'(\R) \right)}\cdot \left| \frac{\om_{E_h}}{\phi^* \om_{E_h'}} \right|,
\end{equation}
where $\om_E$ denotes the invariant differential of a minimal model of an elliptic curve $E$. First we observe that $$\# \textup{ker} \left( \phi\colon E_h(\R) \to E_h'(\R) \right)=\#\coker \left( \phi\colon E_h(\R) \to E_h'(\R) \right)=2.$$ Recall that the models \eqref{eqEh2} and \eqref{eqEh'} are minimal. Using the explicit form of $\phi$ (see, for instance, \cite[Proposition~X.4.9]{BibSilv}) and writing $a:=432-36(h-6)^2$ and $b:=-108h^3(h-8)$, we compute
\begin{align*}
    \phi^*\om_{E_h'} &= \phi^*\left(\frac{dX}{2Y}\right)=\frac{d(x^{-2}(x^3+ax^2+bx))}{2y(bx^{-2}-1)} \\
    &=\frac{(1-bx^{-2})dx}{2y(bx^{-2}-1)} = -\om_{E_h}.
\end{align*}
Thus \eqref{eqreaper} yields that $\Om_{E_h}=\Om_{E_h'}$. Now one applies Theorem \ref{thmCass} to infer that $\#\Selphi(E_h/\Q)=\#\Selphih(E_h'/\Q)=2$, from which the result follows immediately.

\end{proof}

\begin{corollary} \label{corrank0}
The Mordell-Weil groups $E_h(\Q)$, $E_h'(\Q)$ and $\bar{E}_h(\Q)$ have rank zero.
\end{corollary}
\begin{proof}
This is a consequence of the formula 
(see, for example, \cite[page 91]{BibTateSilv})
$$2^{\text{rank}_\Q(E_h)}=\frac{\# \im(\a) \# \im(\a')}{4}.$$  
Since $E_h'$ and $\bar{E}_h$ are isogenous to $E_h$, these also have rank zero over $\Q$.
\end{proof}

\subsection{Descent by \texorpdfstring{$3$}\ -isogeny} \label{subsecEh3-desc}

We now compute $\Selpsi(E/\Q)$. Corollary \ref{corrank0} guarantees that $E_h(\Q)$ has rank zero and our choice of Equation \eqref{eqEh} shows that $E_h(\Q)$ does not have $3$-torsion. Thus non-zero elements of $\Selpsi(E_h/\Q)$ cannot come from points on $\bar{E}_h(\Q)/\phi(E(\Q))$. 

\begin{theorem} \label{thmshaEh}
$\Selpsi(E/\Q)$ consists of $9$ elements and is isomorphic to $\Sha(E_h/\Q)[3]$.
\end{theorem}
\begin{proof}
The lower bound for $\#\Selpsi(E/\Q)$ comes from Theorem \ref{thmCass}. The only quantities yet unknown are the real periods of $E$ and $\bar{E}$. We again use \eqref{eqreaper}.
Both the real kernel and the real cokernel of $\psi$ are trivial, so we need only compute $\om_{E_h}$ and $\psi^* \om_{\bar{E}_h}$. A minimal model of $\bar{E}_h$ is given by $y^2=x^3+8(3x-(h-2)(h-6)^2)^2$. 
Now write $A:=8$ and $B:=(h-2)(h-6)^2$. The explicit form \eqref{eqpsiAB} of $\psi$ then gives
\begin{align*}
    \psi^*\om_{\bar{E}_h} = \psi^*\left(\frac{d\x}{2\y} \right)
    &= \frac{d\left( 3^{-3}x^{-2}(6y^2+6AB^2-3x^3-2Ax^2) \right) }{2 \cdot 3^{-3} y(8AB^2-x^3-4ABx)x^{-3} }& \\
    &= \frac{x^3 d\left(x^{-2} (6\left(x^3+A(x-B)^2\right)+6AB^2-3x^3-2Ax^2) \right) }{2 y(8AB^2-x^3-4ABx)} \\
    &= \frac{x^3(3+12ABx^{-2}-24AB^2x^{-3})dx}{2 y(8AB^2-x^3-4ABx)} \\
    &= \frac{-3dx}{2y} = -3\om_{E_h}.
\end{align*}
Equation \eqref{eqreaper} in the present situation reads
$$\frac{\Om_{E_h}}{\Om_{\bar{E}_h}} = \frac{1}{1}\cdot \frac{1}{3}$$
so Theorem \ref{thmCass} yields
$$\frac{\# \Selpsih(\bar{E}_h/\Q)}{\# \Selpsi (E_h/\Q)} = \frac{\# \bar{E}_h(\Q)[\psih] \Om_{E_h} \prod_{p} c_{E_h,p} } { \# E_h(\Q)[\psi] \Om_{\bar{E}_h} \prod_{p} c_{\bar{E}_h,p} } = \frac{ 1 \cdot 1 \cdot 16} {1 \cdot 3 \cdot 48} = \frac{1}{9} .$$
As a consequence $9 | \# \Selpsi(E_h/\Q)$.

An upper bound on $\# \Selpsi(E_h/\Q)$ is obtained by arithmetic in $L=\Q(\sqrt{2})=\Q(\bar{E}_h[\psih])$. We compute $\Lstar$, the elements of $\Lthree$ that only ramify at the bad primes and have cube norm. Recall that by Lemma \ref{lemcube} the inert primes and the principally ramified primes of $\mathcal{O}_L$ cannot contribute. The prime $2$ is ramified in $\mathcal{O}_L=\Z[\sqrt{2}]$, a principal ideal domain. Since $2$ is non-square in $\F_3$, $\F_h$, $\F_{h-6}$ and $\F_{h-8}$, these primes are inert. Only the prime $h-2=(k+l\sqrt{2})(k-l\sqrt{2})$ is split. Thus $\Lstar$ is generated by $(k+l\sqrt{2})^2(k-l\sqrt{2})$ and the fundamental unit $1+\sqrt{2}$, but this gives $\#\Lstar =9$. We conclude
$$9 \leq \# \Selpsi(E/\Q) \leq \#\Lstar \leq 9.$$ 

Since $E_h(\Q)$ and $\bar{E}_h(\Q)$ have rank zero and no $3$-torsion, $\bar{E}_h(\Q)/\psi(E(\Q))$ is trivial. The exact sequence
\begin{equation*} \label{eqselsha3}
    0 \to \bar{E_h}(\Q)/\psi(E_h(\Q)) \to \Selpsi(E_h/\Q) \to \Sha(E_h/\Q)[\psi] \to 0
\end{equation*}
then implies $\Selpsi(E_h/\Q) \cong \Sha(E_h/\Q)[\psi]$. Finally, since $\Selpsih(E_h'/\Q)$ is trivial by Theorem \ref{thmCass}, we conclude $\Sha(E_h/\Q)[\psi] \cong \Sha(E_h/\Q)[3]$, which finishes the proof.
\end{proof}

\begin{corollary}
Assume $h$ is a positive integer that satisfies the following conditions:
\begin{itemize}
    \item $h \equiv 3 \bmod 8$.
    \item $h$, $h-2$, $h-6$ and $h-8$ are prime numbers.
\end{itemize}
Then the plane cubic
\[
    C_{1+\sqrt{2}}^h\colon 3w^2z+2z^3+w^3+6wz^2+2(h-2)^2(h-8)=
     12z^2-6w^2
\]
has a point over every completion of $\Q$, but not over $\Q$ itself.
\end{corollary}
\begin{proof}
The equation for $C_{1+\sqrt{2}}^h$ is an application of Theorem~\ref{thmhomspeq} with $t=1 +\sqrt{2}$. We use a minimal model of $\bar{E}_h$ with $\bar{A}=72$ and $\bar{B}=(h-2)^2(h-8)/3$ to reduce the size of the coefficients. Theorem~\ref{thmshaEh} asserts that these curves represent non-trivial elements of the Tate-Shafarevich group, meaning that they violate the Hasse principle.
\end{proof}

Observe that $h=19$ satisfies the conditions, hence the homogeneous space
\[
    C_{1+\sqrt{2}}^{19}\colon 3w^2z+2z^3+w^3+6wz^2+6358=12z^2-6w^2 
\]
violates the Hasse principle. By Magma \cite{BibMagma}, the values $h<120 000$ (besides $h=19$) that satisfy our conditions are:
\[3259,  5659, 15739, 21019, 55339, 67219, 69499, 79699, 88819, 99139, 116539, 119299.
\]

Not surprisingly, the set $\{E_h\}$ gives a sparser set of elements of order $3$ in a Tate-Shafarevich group than the list in \cite{BibFisherTab}. An advantage is that the examples coming from $\{E_h\}$ are easy to generate; one only needs the congruence $h \equiv 3 \bmod 8$ and the primality of $h$, $h-2$, $h-6$ and $h-8$. 

\section*{Acknowledgements}
It is our pleasure to thank Steffen M\"{u}ller and Damiano Testa for their interest in this work, and for valuable suggestions.

\bibliography{mainIntegersFormat}
\bibliographystyle{integers}

\end{document}